\documentclass[11pt]{article}
\usepackage{amsfonts,amsthm,hyperref}
\usepackage[ruled,vlined,linesnumbered]{algorithm2e}
\usepackage{a4wide}
\usepackage[usenames,dvipsnames]{color}
\usepackage{amssymb,amsmath,eepic,graphics,enumerate}
\usepackage{graphicx}
\usepackage{epsf}
\usepackage{epsfig}
\usepackage{graphics}
\usepackage{cite}
\usepackage[english]{babel}
\usepackage{color,cleveref}
\usepackage{amsfonts,amsthm,amssymb,amsmath}
\usepackage{a4wide}

\newcommand{\cw}{\mathrm{cw}}
\newcommand{\nlcw}{\mathrm{nlcw}}
\newcommand{\tw}{\mathrm{tw}}
\newcommand{\diam}{\mathrm{diam}}
\newcommand{\dist}{\mathrm{dist}}
\newtheorem{theorem}{Theorem}[section]
\newtheorem{lemma}[theorem]{Lemma}
\newtheorem{corollary}[theorem]{Corollary}
\newtheorem{proposition}[theorem]{Proposition}

\newtheorem{definition}{Definition}
\newtheorem{example}{Example}

\newtheorem{problem}{Problem}
\theoremstyle{definition}

\author{Flavia Bonomo\thanks{Depto.~de Computaci\'on, FCEyN, Universidad de Buenos Aires and IMAS-CONICET, Argentina. Email: \texttt{fbonomo@dc.uba.ar}}
\and
Luciano N.\ Grippo\thanks{Instituto de Ciencias, Universidad Nacional de General Sarmiento, Los Polvorines, Argentina. Email: \texttt{lgrippo@ungs.edu.ar}}
\and
Martin Milani\v c\thanks{University of Primorska, UP IAM and UP FAMNIT, Koper, Slovenia. Email: \texttt{martin.milanic@upr.si}.}
\and Mart\'in D.\ Safe\thanks{Instituto de Ciencias, Universidad Nacional de General Sarmiento, Los Polvorines, Argentina. Email: \texttt{msafe@ungs.edu.ar}}}

\date{}

\title{Graph classes with and without powers of bounded clique-width}

\begin{document}
\maketitle

\begin{abstract}
\begin{sloppypar}
We initiate the study of graph classes of power-bounded clique-width,
that is, graph classes for which there exist integers $k$ and $\ell$ such that the $k$-th
powers of the graphs are of clique-width at most $\ell$. We give sufficient and
necessary conditions for this property. As our main results, we characterize graph
classes of power-bounded clique-width within classes defined by
either one forbidden induced subgraph, or by
two connected forbidden induced subgraphs.
We also show that for every positive integer $k$, there exists a graph class such that
the $k$-th powers of graphs in the class form a class of bounded clique-width,
while this is not the case for any smaller power.
\end{sloppypar}
\end{abstract}


\section{Introduction}

{ The two main notions studied in this paper are the notion of graph powers and the notion of clique-width.
Recall that for a positive integer $k$, the {\it $k$-th power of a graph} $G$ is the graph denoted by $G^k$ and obtained from $G$ by adding to it all edges between pairs of vertices at distance at least $1$ and at most $k$.
Graph powers are basic graph transformations with a number of results about their properties in the literature (see, e.g.,~\cite{MR1686154,MR1379114}). The other main notion of the paper, {\it clique-width}, is a graph parameter, denoted by $\cw(G)$, with
many algorithmic applications when bounded by a constant (see, e.g.,~\cite{MR1739644,MR2379082,MR1973174,MR2272239,MR1948213,MR2274712,MR3197779,MR2080095}).
We study these two notions in the framework of graph classes, that is, sets of graphs closed under isomorphism, paying particular attention to
hereditary graph classes.

For a graph class ${\cal G}$ and a positive integer $k$, the {\em $k$-th power of ${\cal G}$} is the set ${\cal G}^k$ of all $k$-th powers of graphs in ${\cal G}$. The fact that several graph algorithmic problems can be expressed in terms of graph powers (see, e.g.,~\cite{DBLP:conf/isaac/BrandstadtLR12,MR3065112, MR2288329,MR1605685,MR1607726,MR2080095}) motivates the study of graph classes the $k$-th power of which has small clique-width, where $k$ is a fixed positive integer. More specifically, one can study the following properties of graph classes capturing, on a coarse scale,
various dependencies regarding the behavior of the clique-width with respect to graph powers, where,
for a graph class ${\cal G}$, the {\it clique-width} of ${\cal G}$ is defined as $\cw({\cal G}) = \sup\{\cw(G)\mid G\in {\cal G}\}$:
\begin{itemize}
  \item For a pair of positive integers $k$ and $\ell$, we say that ${\cal G}$ is of {\em $(k,\ell)$-power-bounded clique-width}
  if $\cw({\cal G}^k)\le \ell$.
  \item For a positive integer $k$, we say that  ${\cal G}$ is of {\em $(k,\ast)$-power-bounded clique-width}
    if $\cw({\cal G}^k)\le \ell$
    for some positive integer $\ell$.
    \item For a positive integer $\ell$,
  we say that  ${\cal G}$ is of {\em $(\ast,\ell)$-power-bounded clique-width}
  if $\cw({\cal G}^k)\le \ell$
    for some positive integer $k$.
    \item Finally, we say that ${\cal G}$ is of {\em $(\ast,\ast)$-power-bounded clique-width} (or simply:
 {\em of power-bounded clique-width}) if
  $\cw({\cal G}^k)\le \ell$
    for some pair of positive integers $k$ and $\ell$.
\end{itemize}

In the above terminology, a graph class ${\cal G}$ is of bounded clique-width if and only if it is of $(1,\ast)$-power-bounded clique-width. Thus, the above framework can be seen as a graph-powers-oriented extension of the notion of graph classes of bounded clique-width.
We should emphasize that, while characterizing graph classes with respect to boundedness / unboundedness of their clique-width is important for algorithmic purposes, this is not a simple task, as proving lower bounds on the clique-width of a given graph class can be quite challenging.
For instance, while it is known that the class of $H$-free graphs is of bounded clique-width if and only if $H$ is an induced subgraph of the $4$-vertex path~\cite{DBLP:conf/lata/DabrowskiHP15}, a complete dichotomy
for graph classes defined by {\it two} forbidden induced subgraphs, say $H$ and $H'$, is still not known, even in the case when both $H$ and $H'$ are connected~\cite{CIACDabrowskiP14a, DBLP:conf/lata/DabrowskiHP15}. Moreover, given a graph $G$ and an integer $k$, it is {\sf NP}-complete to determine if the clique-width of $G$ is at most $k$~\cite{MR2519936}, while for fixed values of $k$ polynomial-time algorithms are known only for $k \le 3$~\cite{MR2901093}.

We thus propose the study of graph classes of $(k,\ell)$-power-bounded clique-width, which, as indicated above, can be useful for algorithmic purposes also for $k>1$. The difficulties of understanding the corresponding graph classes for fixed values of $k$ and $\ell$ motivate the
introduction of the more relaxed properties of $(k,\ast)$-, $(\ast, \ell)$, and $(\ast, \ast)$-power-bounded clique-width.
We expect that relaxing one or both of the two parameters to be unconstrained might lead to more tractable cases in terms of proving dichotomy results. This seems to be indeed the case, as certified for instance by the complete characterization of graph classes of
power-bounded clique-width within classes defined by two connected forbidden induced subgraphs, which we prove in this paper (Theorem~\ref{thm:bigenic classes-characterization}).
At the same time, we expect that research leading to such results will also lead to discoveries of new structural properties of the graph classes under consideration.

We now summarize our main results and connect them with some known results from the literature. We focus mainly on the last, most relaxed property, that is, on graph classes of power-bounded clique-width. Several of our results also have implications for the other three properties.

First, we observe that several well-known graph classes, including grids, bipartite permutation graphs, unit interval graphs, and hypercube
graphs, are of power-unbounded clique-width. In particular, this implies that none of these graph classes is of $(k,\ell)$-,
$(k,\ast)$-, or $(\ast,\ell)$-power-bounded clique-width, for any positive integers $k$ and $\ell$, thus
strengthening the known fact that these graph classes are of unbounded clique-width.

Second, for every positive integer $k$, we construct a graph class ${\cal G}$ such that
the power class ${\cal G}^k$ is of bounded clique-width, while this is not the
case for any smaller power. This implies that the families of $(k,\ast)$-power-bounded classes are all pairwise distinct
and further motivates the study of these properties.

Third, we prove a sufficient condition for power-boundedness of the clique-width, generalizing
the simple observation that every graph class of bounded diameter is of power-bounded
clique-width. 
Informally speaking, the condition states that for every class ${\cal G}$ of graphs of bounded diameter, the
class of graphs arising from graphs in ${\cal G}$ by subdividing (arbitrarily many times) a
bounded number of edges is of power-bounded clique-width.

Finally, using the above condition, we develop our main result:
a complete characterization of graph classes of
power-bounded clique-width within classes defined by two connected forbidden induced subgraphs (Theorem~\ref{thm:bigenic classes-characterization}).
As remarked above, this result contrasts with the case of graph classes of bounded clique-width, that is, of $(1,\ast)$-power-bounded clique-width, where a dichotomy for graph classes defined by two connected induced subgraphs is (at the time of this writing) still not known.
We also characterize graph classes of power-bounded clique-width within hereditary graph classes defined by a single forbidden induced subgraph
(Theorem~\ref{thm:monogenic}), thus extending the analogous characterization for graph classes of bounded clique-width.}

The rest of the paper is structured as follows. In Section~\ref{sec:notation} we review the necessary preliminaries and basic definitions.
In Section~\ref{sec:basic} we formally introduce the central notion of the paper,
the power-(un)boundedness of the clique-width, obtain initial insight into this notion, and develop
results that we use in later sections.
In Section~\ref{sec:split}, we construct graph classes of power-bounded clique-width
that require taking arbitrarily large powers in order to produce a graph class of bounded clique-width.
In Section~\ref{sec:sufficient}, we prove a sufficient condition for power-bounded clique-width.
Section~\ref{sec:hereditary} is devoted to results about power-boundedness of the clique-width in hereditary graph classes.
We conclude the paper with a discussion in Section~\ref{sec:conclusion}.

\section{Preliminaries}\label{sec:notation}

All graphs in this paper are finite, simple and undirected. Graph terminology not defined here can be found in
\cite{West}.

\bigskip
\noindent {\bf Graphs and graph classes.} Given a graph $G$, an
{\it independent set} in $G$ is a set of pairwise non-adjacent
vertices, and a {\it clique} is a set of pairwise adjacent
vertices. Given two graphs $G$ and $H$, graph $H$ is said to be an
{\it induced subgraph} of $G$ if it can be obtained from $G$ by a
sequence of vertex deletions, a {\it subgraph} of $G$ if it can be
obtained from $G$ by a sequence of vertex and edge deletions, and
a {\it minor} of $G$ if it can be obtained from $G$ by a sequence
of vertex deletions, edge deletions, and edge contractions. If an
induced subgraph of $G$ is isomorphic to a graph $H$, we say that
$G$ {\it contains} $H$ (as an induced subgraph). For a subset of
vertices $X\subseteq V(G)$, we will denote by {$G-X$ the graph
obtained from $G$ by deleting from it the vertices in $X$, and} by
$G[X]$ the subgraph of $G$ induced by $X$, that is, $G[X] =
G-(V(G)\setminus X)$. For two vertices $x,y$ in a connected graph
$G$, we denote by $\dist_G(x,y)$ the {\it distance between $x$ and
$y$}, that is, the length (number of edges) of a shortest
$x,y$-path in $G$. The {\it diameter} of a connected graph $G$ is
defined as $\diam(G) = \max_{x,y\in V(G)}\dist_G(x,y)$, and we
define the diameter of a disconnected graph $G$ to be the maximum
diameter of a connected component of $G$. By $P_n$, $C_n$, and
$K_n$, we denote the path, the cycle, and the complete graph on
$n$ vertices, respectively. For two vertex-disjoint graphs $G_1$
and $G_2$, the {\em disjoint union} of $G_1$ and $G_2$ is the
graph $(V(G_1)\cup V(G_2),E(G_1)\cup E(G_2))$. The disjoint union
of $k$ graphs isomorphic to a graph $H$ will be denoted by $kH$.
The {\it complement} of a graph $G = (V,E)$ is the graph
$\overline{G}$ with the same vertex set as $G$, in which two
distinct vertices are adjacent if and only if they are
non-adjacent in $G$. A graph is said to be {\it co-connected} if
its complement is connected.
The {\it treewidth} {of a graph $G$ is} denoted by $\tw(G)$.
We refer to~\cite{MR1647486} for several equivalent
characterizations.

A {\it graph class} is a set of graphs {that is} closed under isomorphism.
Given a
graph class ${\cal G}$, we say that ${\cal G}$ is {\em of bounded degree} if
$\sup\{\Delta(G)\mid G\in {\cal G}\}<\infty$, {and {\em of bounded diameter} if
$\diam({\cal G})= \sup\{\diam(G)\mid G\in {\cal G}\}<\infty$.}
For a set ${\cal F}$ of
graphs, we say that a graph $G$ is {\it ${\cal F}$-free} if no
induced subgraph of $G$ is isomorphic to a member of ${\cal F}$. Similarly, for
a graph $H$, we say that $G$ is $H$-free if it is $\{H\}$-free.
The set of all ${\cal F}$-free graphs will be denoted by ${\it
Free}({\cal F})$. A graph class is {\it hereditary} if it is
closed under taking induced subgraphs. {A} graph class ${\cal G}$ is hereditary if
and only if ${\cal G}={\it Free}({\cal F})$ for some set ${\cal
F}$ of graphs. An important family of hereditary {graph} classes is the
family of minor-closed {classes} (i.e., graph classes closed under
taking minors). For graph classes not defined in this paper, we
refer to~\cite{MR1686154}.

\bigskip
\noindent{\bf Clique-width.}~\cite{MR1217156} The {\em clique-width} of a graph $G$, denoted $\cw(G)$, is the
minimum number of different labels needed to construct a vertex-labeled graph isomorphic to
$G$ using the following four operations: (i) Creation of a new vertex $v$ with label $i$; (ii) Disjoint union of two labeled graphs $G$ and $H$; (iii) Joining by an edge each vertex with label $i$ to each vertex with label $j$ (for some pair of different labels $i$ and $j$); (iv)
{R}elabeling each vertex with label $i$ with label $j$. {E}very graph can be
built using the above four operations.

Given a graph class ${\cal G}$, the clique-width of ${\cal G}$ is
$\cw({\cal G})= \sup\{\cw(G)\mid G\in {\cal G}\}\,.$ We say that
${\cal G}$ is {\em of bounded clique-width} if $\cw({\cal
G})<\infty$ (and {\em of unbounded clique-width}, otherwise).

We will often make use of the following basic property of the clique-width.

\begin{proposition}[Johansson~\cite{MR1676494}, Courcelle-Olariu~\cite{MR1743732}]\label{prop:induced}
If $H$ is an induced subgraph of {a graph} $G$ then
$\cw(H)\le\cw(G)$.
\end{proposition}

We will make use of the following fact in Sections~\ref{sec:split}
and~\ref{sec:hereditary}.

\begin{proposition}[Kami{\'n}ski {et al.}~\cite{MR2536473}]\label{prop:complementation}
If ${\cal G}$ is a graph class of unbounded clique-width, then the
class of graphs obtained from graphs in ${\cal G}$ by applying a
constant number of
operations of
replacing an induced subgraph of $G$ with its complement
is also of unbounded clique-width.
\end{proposition}

\bigskip
\noindent{\bf Modules.} A subset $M$ of vertices in a graph $G$ is said to be a {\it module}
if every vertex $v\in V(G)\setminus M$ is either adjacent to all vertices in $M$, or non-adjacent to all vertices of $M$.
A module is said to be {\it trivial} if $M= V$ or $|M|\le 1$, and a graph $G$ is
{\it prime} if it does not contain any nontrivial module.
Given a partition $\Pi$ of the vertex set of a graph $G$ into modules,
the {\em quotient graph of $G$ with respect to $\Pi$} is defined as the
graph obtained from $G$ by
{replacing}
the sets in $\Pi$ {with} single vertices {and connecting two vertices by en edge if and only if the corresponding
two sets in $\Pi$ are connected by an edge in $G$}.
If $G$ is connected and co-connected, then its vertex set admits a unique partition into pairwise disjoint maximal modules~(see, e.g.,~\cite{MR549499}).
Moreover, the corresponding quotient graph is always prime.

\begin{sloppypar}
\begin{proposition}[Courcelle-Olariu~\cite{MR1743732}]\label{prop:cwd-prime2}
For every graph $G$
we have $\cw(G) = \max\{\cw(H)\mid H$ is a prime induced subgraph
of $G\}$.
\end{proposition}
\end{sloppypar}

{The following proposition shows that for hereditary graph classes, the power-boundedness of the clique-width depends only on the prime graphs in the class.

\begin{proposition}\label{lem:power-bdd-prime}
Let ${\cal G}$ be a hereditary graph class and let ${\cal G}'$ be
the set of all prime graphs in ${\cal G}$. Then, for every
{$k\ge 1$}, the graph class ${\cal G}^k$ is of bounded
clique-width if and only if $({\cal G}')^k$~is of bounded
clique-width.
In particular, ${\cal G}$ is of power-bounded clique-width if and only if ${\cal G}'$ is of
power-bounded clique-width.
\end{proposition}

\begin{proof}
Since ${\cal G}'\subseteq {\cal G}$, we have $({\cal G}')^k\subseteq {\cal G}^k$.
Therefore, if ${\cal G}^k$ is of bounded clique-width, then so is $({\cal
G}')^k$. Suppose now that $\cw(({\cal G}')^k)\le \ell$ for some $k\ge 1$ and $\ell \ge 2$.
By induction on the number of vertices, we will prove that for every $G\in {\cal G}$, we have $\cw(G^k)\le\ell$.
If $G$ is disconnected, then so is $G^k$, and we can assume inductively that $\cw(H)\le \ell$ holds for
every connected component $H$ of $G^k$, which implies the desired
inequality for $G^k$. If the complement of $G$ is disconnected,
then $\diam(G)\le 2$, hence $G^k$ is complete and the result
follows. Now, let $G$ be a connected co-connected graph in ${\cal G}$,
and let $Q$ be the quotient graph of $G$ with respect to the
partition of $V(G)$ into maximal modules. Then, the graph $G^k$ is
isomorphic to the graph obtained from the graph $Q^k$ by
substituting a clique of size $|M|$ for each vertex $M$ of $Q$. In
particular, every prime induced subgraph of $G^k$ is isomorphic to
an induced subgraph of $Q^k$. Applying
Proposition~\ref{prop:cwd-prime2} twice, we obtain
$\cw(G^k) \le \max\{\cw(H)\mid H$ is a prime induced subgraph
of $Q^k\}= \cw(Q^k)\le \ell\,,$ as claimed.
\end{proof}}

\section{The definition, basic properties, and examples}\label{sec:basic}

\begin{sloppypar}
{In this section, we obtain some initial insight into power-(un)boundedness of the clique-width.
Several of the results developed in this section will be used in later sections.

The central notion of the paper is introduced in the following.}
\end{sloppypar}

\begin{definition}\label{def:main}
A graph class ${\cal G}$ is said to be of {\em power-bounded
clique-width} if there exists a positive integer $k$ such that
${\cal G}^k$ is of bounded clique-width. {If no such $k$ exists,
we say that} ${\cal G}$ is of {\em power-unbounded clique-width}.
\end{definition}

{In other words}, a graph class ${\cal G}$ is of power-bounded
clique-width if there exists a pair of positive integers $k$ and
$\ell$ such that for every $G\in {\cal G}$, we have
$\cw(G^k)\le\ell$.
For a {graph} class ${\cal G}$ of power-bounded clique-width, we
denote by $\pi({\cal G})$ the smallest positive integer $k$ such
that ${\cal G}^k$ is of bounded clique-width. Clearly, $\pi({\cal
G}) = 1$ if and only if ${\cal G}$ is of bounded clique-width.
If ${\cal G}$ is of power-unbounded clique-width, then $\pi({\cal G})$ is
defined to be $\infty$.
In some arguments in the paper, we will use the obvious
fact that if ${\cal G}\subseteq {\cal H}$, then $\pi({\cal G})\le
\pi({\cal H})$. Consequently, if ${\cal G}\subseteq {\cal H}$ and
${\cal H}$ is of power-bounded clique-width, then so is ${\cal
G}$.

In the next proposition, we collect some basic properties of the family of graph classes of power-bounded clique-width.
In particular, the family is closed under taking powers and contains all graph classes of bounded diameter.

{
\begin{proposition}\label{prop:pw-bdd-cwd-powers}
Let ${\cal G}$ be a graph class.
Then:
\begin{enumerate}
  \item If $\cw({\cal G}) \le \ell$ for some $\ell\ge 1$, then for every $k\ge 1$ we have $\cw({\cal G}^k) \le 4(k+1)^{\ell}$.
  \item For every $k\ge 1$, we have $\pi({\cal G}^k)\le \pi({\cal G})\le \diam({\cal G})\,.$
\end{enumerate}
\end{proposition}

\begin{proof}
Suppose that $\cw({\cal G}) \le \ell$ for some positive integer $\ell$, and let $k$ be a positive integer.
Denoting by $\nlcw(G)$ the NLC-width of a graph $G$, every graph $G$ with
$\nlcw(G)\le \ell$ satisfies $\nlcw(G^k)\le 2(k+1)^\ell$~\cite{MR2354335}, and
for every graph $G$, we have $\nlcw(G)\le \cw(G)\le 2\cdot\nlcw(G)$~\cite{MR1676494}.
Therefore, if $\cw(G) \le \ell$, then $\nlcw(G)\le \ell$, and
consequently $\cw(G^k) \le 2\cdot\nlcw(G^k) \le 4(k+1)^{\ell}\,.$
This implies that for every $G\in {\cal G}$, we have $\cw(G^k) \le 4(k+1)^{\ell}$, proving
the first part of the proposition.

For the second part of the proposition, we first show the inequality
$\pi({\cal G}^k)\le \pi({\cal G})$. Suppose that $p = \pi({\cal G})$ is finite (otherwise, there is nothing to show).
Then, $\ell = \cw({\cal G}^p)$ is finite.
{We have already proved that} for every graph $G\in {\cal G}^p$, we
have $\cw(G^k)\le  4(k+1)^{\ell}$. Consequently, for every graph
$G\in {\cal G}$, we have $\cw((G^p)^k)\le 4(k+1)^{\ell}$.
Observe that the graph $(G^p)^k$ is equal to the
graph $G^{pk}$, and, by symmetry, to the graph $(G^k)^p$. This
implies that for every graph $G\in {\cal G}$, we have
$\cw((G^k)^p)\le 4(k+1)^{\ell}$, and thus $\pi({\cal G}^k)\le p$.

Finally, we show that $\pi({\cal G})\le \diam({\cal G})$.
Suppose that $k = \diam({\cal G})$ is finite (otherwise, there is nothing to show).
Then, for every $G\in {\cal G}$, the graph $G^k$ is a disjoint union of complete graphs, and hence
$\cw(G^k)\le 2$. Consequently, $\cw({\cal G}^{\diam({\cal G})})\le
2$ and the claimed inequality
follows.
\end{proof}}

We continue with the observation that for graphs of bounded degree
and proper minor-closed graph classes (that is, minor-closed graph classes excluding at least one minor),
power-bounded clique-width is equivalent to bounded clique-width and
bounded treewidth. This result will be used in Section~\ref{sec:unbounded}.

\begin{proposition}\label{prop:bdd-deg}
Let ${\cal G}$ be a graph class
that is either of bounded degree or minor closed. Then, the following are equivalent:
\begin{enumerate}
\item $\cal G$ is of power-bounded clique-width.
\item $\cal G$ is of bounded clique-width.
\item $\cal G$ is of bounded treewidth.
\end{enumerate}
\end{proposition}

\begin{proof}
{If ${\cal G}$ is of bounded treewidth, then it is} of
bounded clique-width~\cite{MR2148860}, hence also of power-bounded clique-width.
{Therefore, we only need to show that power-bounded clique-width implies bounded treewidth.}

\begin{sloppypar}
{Assume that $\Delta(G)\le d$ for all $G\in {\cal G}$, and that
there are $k,\ell\ge 1$ such that such that $\cw({\cal G}^k)\le \ell$.}
{Observe: $\Delta(G^k)\le d\cdot \sum_{i =
0}^{k-1}(d-1)^i\le d^{k+1}$}. Courcelle and Olariu showed
in~\cite{MR1743732} that there exists a function $f$
such that for every graph $G$, we have
$\tw(G)\le f(\Delta(G),\cw(G))$.
This implies the existence of a function $g$
that is non-decreasing in each component such that
$\tw(G)\le g(\Delta(G),\cw(G))$ holds for all graphs.
Therefore, since adding edges cannot decrease the treewidth, we have for every $G \in {\cal G}$:
$$\tw(G) \le \tw(G^k) \le  g(\Delta(G^k), \cw(G^k))\le  g(d^{k+1},\ell)\,.$$
Thus, ${\cal G}$ is of bounded treewidth.
\end{sloppypar}

{Finally, let ${\cal G}$ be a minor-closed graph class of power-bounded clique-width.
Since the class of $n\times n$ grids (see Section~\ref{sec:unbounded} for the definition) is of unbounded treewidth (see, e.g.,~\cite[Corollary
89]{MR1647486}), the above implies that the class of grids
is also of power-unbounded clique-width. Therefore, ${\cal G}$
excludes some grid $G$.} Since ${\cal G}$ is {minor closed}, no graph
in ${\cal G}$ has a minor isomorphic to $G$. {Since every graph class excluding a
fixed planar graph as a minor is of bounded treewidth~\cite{MR855559}, the conclusion follows.}
\end{proof}

\subsection{Examples of graph classes of power-unbounded clique-width}\label{sec:unbounded}

{For an integer $n\ge 1$, the {\it $n\times n$ grid} is the graph with vertex set
$\{1,\ldots, n\}^2$, in which two vertices $(i,j)$ and $(k,\ell)$
are adjacent if and only if $|i-k|+|j-\ell| = 1$.}

\begin{example}\label{ex:grids}
For every $k\ge 1$, the set of graphs obtained from grids by replacing each edge with a path with $k$ edges
 is of power-unbounded clique-width.
\end{example}

Indeed, let $G_{n,k}$ be the graph obtained from the $n\times n$ grid by
replacing each edge with a path with $k$ edges. Since the $n\times
n$ grid $G_{n,1}$ is a minor of $G_{n,k}$, and $n\times n$ grids are of
unbounded treewidth, the set of graphs $\{G_{n,k}\mid
n\ge 1\}$ is also of unbounded treewidth. (It is well known that
if $H$ is a minor of $G$, then $\tw(H)\le\tw(G)$, see, e.g.,~\cite[Lemma 16]{MR1647486}.)
The conclusion now follows from
Proposition~\ref{prop:bdd-deg}.

{The {\it girth}} of a graph $G$ is defined as the shortest length of a cycle in $G$
(or infinity if $G$ is acyclic). The next example immediately follows as a consequence of Example~\ref{ex:grids}.

\begin{example}\label{cor:girth}
For every $k\ge 3$, the class of graphs of girth at least $k$ is of power-unbounded clique-width.
\end{example}

In the next proposition, we show that bipartite permutation graphs, path powers, unit interval graphs, and hypercube graphs are of
power-unbounded clique-width. {A} graph $G$ is a {\it bipartite permutation graph} if it is both bipartite and permutation,
where a graph $G = (V,E)$ is {\it bipartite} if its vertex set can be partitioned into two independent sets, and
{\it permutation} if there exists a permutation $\pi = (\pi_1,\ldots, \pi_n)$ of the set $\{1,\ldots, n\}$
where $V = \{v_1,\ldots, v_n\}$ such that $v_iv_j\in E$ if and only if $(\pi_i-\pi_j)(i-j)<0$.
The class of path powers is defined as ${\cal P}^+ = \{{(P_n)^k}\mid n\ge 1, k\ge 1\}$.
A graph $G$ is a {\it unit interval graph} if it is the intersection graph of a collection of unit
intervals on the real line. Bipartite permutation graphs and unit interval graphs were shown by Lozin~\cite{MR2854788} to be minimal graph classes of unbounded clique-width (in the sense that every proper hereditary subclass of either unit interval or bipartite permutation graphs is of bounded clique-width). For an integer $d\ge 1$, the {\it $d$-dimensional hypercube graph} is the graph $Q_d$
with vertex set given by all $2^d$ binary sequences of length $d$, in which two vertices
are adjacent if and only {their sequences differ in exactly one coordinate}.

\begin{proposition}\label{prop:path-powers}
Each of the following graph classes is of power-unbounded clique-width:
bipartite permutation graphs, path powers, unit interval graphs, hypercube graphs.
\end{proposition}

\begin{proof}
First, we consider bipartite permutation graphs. Let $B_n$ be the graph {whose} set of vertices is defined by $V(B_n)=\{(i,j)\mid \,0\leq i,j\leq n-1\}$ and {whose} set of edges is defined by $E(B_n)=\{(i_1,j),(i_2,j+1)\mid\,0\leq j\leq n-2, 1\leq i_1\leq n-1, 0\leq i_2\leq i_1-{1}\}$ (see Fig.~\ref{fig:GnLn}). It follows from~\cite[Theorem~1]{MR1973244} that $B_n$ is a bipartite permutation graph.

\begin{figure}[!ht]
  \begin{center}
\includegraphics[width=\linewidth]{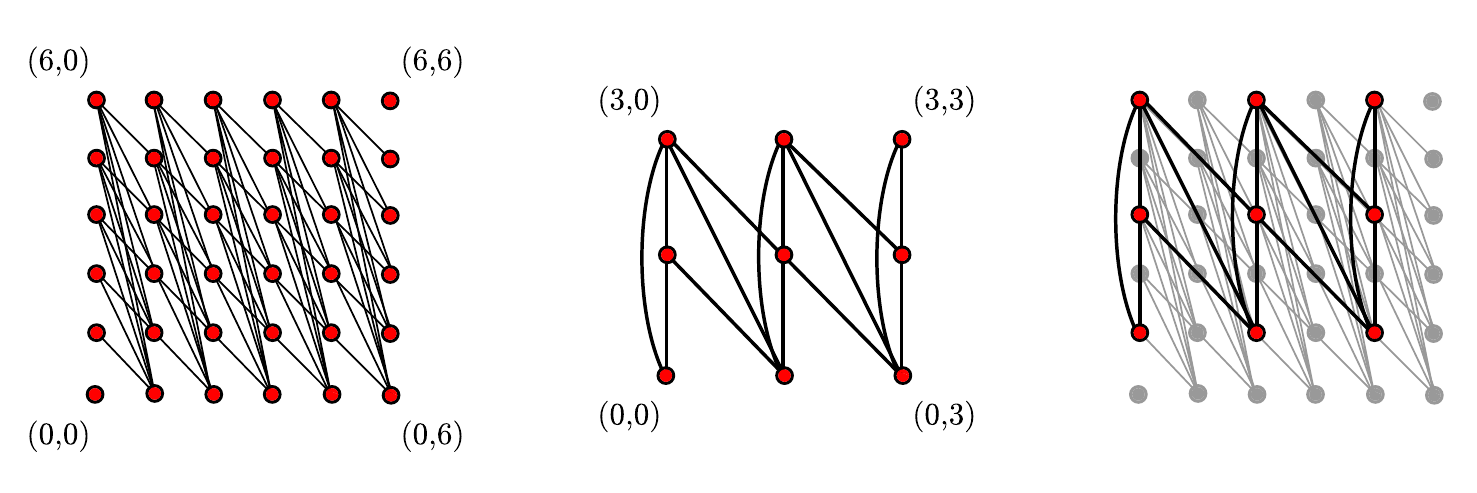}
  \end{center}
\caption{The first two graphs, from left to right, are $B_{6}$ and
$G_3$. The rightmost figure shows that $G_3$ is an induced
subgraph of $B_6^2$ (for the sake of clarity, we do not draw all the edges of
$B_6^2$).} \label{fig:GnLn}
\end{figure}

We denote by $G_n$ the graph obtained from $B_n$ by adding an edge between every two vertices that agree in the second coordinate (see
Fig.~\ref{fig:GnLn} for $G_3$).
Fix an integer $k\ge 2$. The subgraph of $B_{k(n-1)+2}^k$ induced
by $\{(ik+1,jk)\mid \,0\leq i,j\leq n-1\}$ is isomorphic to $G_n$;
in fact, $(ik+1,jk)\mapsto (i,j)$ is an isomorphism between the
two graphs. Since it was proved in \cite{MR1792124} that
$\cw(G_n)\geq n$, {Proposition~\ref{prop:induced} implies that}
$\cw(B_{k(n-1)+2}^k)\geq n$ for every $n\geq 1$ and $k\geq 2$.
This proves that the class of bipartite permutation graphs is of
power-unbounded clique-width.

Now, consider the class of path powers. It was proved in \cite{MR2972287} that for each positive integer $s$ and each
$n\geq (s+1)^2$, we have $\cw(P_n^s)=s+2$.
Therefore, for each pair of {positive} integers $k$ and $N$, there exist two {positive} integers $j$ and $n$ such that the
$k$-th power of the graph $P_n^j\in {\cal P}^+$ has clique-width more than $N$.
Indeed, we can take $j= N$ and $n= (Nk+1)^2$, obtaining
$\cw((P_n^j)^{k})=\cw(P_{(Nk+1)^2}^{{Nk}})=Nk+2>N$.

{Path powers are unit interval graphs \cite[Theorem~9(ii)]{MR2774114}.}

\begin{sloppypar}
Finally, consider the hypercube graphs.
{We will show that for every two positive integers $k$ and $d$, there exists an integer $d'$ such that the $k$-th power of the $d'$-dimensional hypercube $Q_{d'}^k$ contains the $d$-dimensional hypercube $Q_d$ as an induced subgraph.
This will imply the desired conclusion, since hypercube graphs are of unbounded clique-width, which follows from the known facts
that the hypercube graphs are of unbouned treewidth~\cite{MR2204113} and do not contain $K_{3,3}$ as a subgraph, and
the inequality $\tw(G)\le 3(t-1)\cw(G)-1$, which holds for every graph $G$~that does not contain $K_{t,t}$ as a subgraph~\cite{MR1850348}.}

The integer $d'$ can be defined as $d' = dk$. For $x \in V(Q_d) = \{0, 1\}^d$, define $[x]^k$ as the {\it extension} of $x$,
replacing each coordinate of $x$ by $k$ consecutive coordinates of the same value.
We will argue that for all $x,y\in V(Q_d)$, we have
$$xy \in  E(Q_d) \Longleftrightarrow [x]^k[y]^k \in E((Q_{dk})^k)\,,$$
which proves the claimed result.
Suppose first that $xy\in E(Q_d)$. This means that $x$ and $y$ differ in exactly one~coordinate. Hence, $[x]^k$ and $[y]^k$ differ in exactly $k$ coordinates, and therefore $[x]^k[y]^k \in E((Q_{dk})^k)$.
Conversely, suppose that $x$ and $y$ are vertices of $Q_d$ such that
 $[x]^k[y]^k \in E((Q_{dk})^k)$.
 Then, $[x]^k$ and $[y]^k$ differ in at most $k$ coordinates. On the other hand, since $[x]^k$ and $[y]^k$ are distinct vertices of $Q_{dk}$, they differ in at least $k$ coordinates. Hence, they differ in exactly $k$ coordinates, which implies that $x$ and $y$ differ in exactly one coordinate and therefore $xy\in E(Q_d)$.\qedhere
\end{sloppypar}
\end{proof}

\section{Graph classes with arbitrary finite value of $\pi({\cal G})$}\label{sec:split}

In this section, we construct graph classes of power-bounded clique-width that require taking arbitrarily large powers in order to produce a graph class of bounded clique-width. More specifically, we will show that
the value of $\pi({\cal G})$ can be an arbitrary positive integer.
Recall that for a class ${\cal G}$ of power-bounded clique-width,
we denote by $\pi({\cal G})$ the smallest positive integer $k$
such that ${\cal G}^k$ is of bounded clique-width.

Our constructions are based on the class of split graphs. {A {\em split graph} is a graph $S$ that has a {\em split partition},
that is, a pair $(K,I)$ such that $K$ is a clique,
$I$ is an independent set, $K\cup I = V(S)$, and
$K\cap I = \emptyset$. In what follows, we will only consider split partitions $(K,I)$
such that $K$ is a maximal clique in $S$.

Let $S$ be a split graph with a split partition $(K,I)$, and let $k\ge 3$.
We define $S_k$ to be the graph constructed as follows.
Let $K = \{w_1, \dots, w_r\}$ and
$I = \{v_1, \dots, v_s\}$.
\begin{itemize}
  \item {\bf Case 1:} {\it $k\ge 4$ is even.}

In this case, $S_k$ is the graph obtained from $S$ by making $I$ a clique, adding, for each vertex $v_i \in I$, a path $P_i$ of length
$k/2$ having (new) vertices $v_i = v_i^0,v_i^1, \dots, v_i^{k/2}$, and adding, for each vertex $w_j \in K$, a path $Q_j$
of length $k/2-2$ having (new) vertices $w_j = w_j^0,w_j^1, \dots, w_j^{k/2-2}$. (See
Fig.~\ref{fig:S_k} for an example.) 

\begin{figure}[!ht]
  \begin{center}
\includegraphics[width=\textwidth]{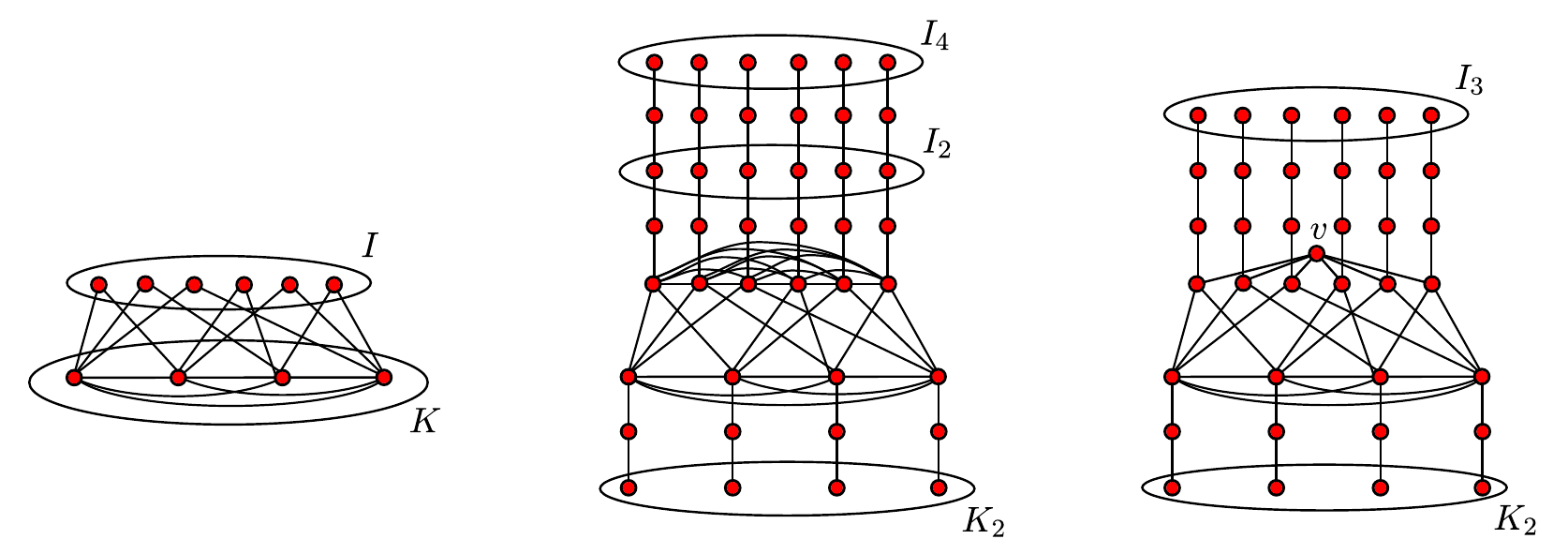}
  \end{center}
\caption{From left to right, graphs $S$, $S_8$ and $S_7$.} \label{fig:S_k}
\end{figure}

\item {\bf Case 2:} {\it $k\ge 3$ is odd.}

In this case, $S_k$ is the graph obtained from $S$ by
adding a new vertex $v$ and making it adjacent to every vertex in $I$,
adding, for each vertex $v_i \in I$, a path $P_i$ of length
$(k-1)/2$ having (new) vertices $v_i = v_i^0,v_i^1, \dots,
v_i^{(k-1)/2}$, and adding, for each vertex $w_j \in K$, a path $Q_j$
of length $(k-3)/2$ having (new) vertices $w_j = w_j^0,w_j^1, \dots, w_j^{(k-3)/2}$.
(See Fig.~\ref{fig:S_k} for an example.) 
\end{itemize}

Note that since we assumed that $K$ is a maximal clique in $S$,
the graph $S_k$ does not depend on the choice of a particular split partition.
A {\it complete split graph} is a split graph with a split partition $(K,I)$ such that every vertex in $K$ is adjacent to every vertex in $I$.

\begin{lemma}\label{lem:unbounded-pi}
For every split graph $S$, every $k\ge 3$ and every $h\ge 1$, we have:
\begin{enumerate}
  \item If $h\ge k$, then $(S_k)^h$ is a complete split graph.
  \item If $h<k$, then $(S_k)^h$ contains $S$ as induced subgraph.
\end{enumerate}
\end{lemma}

\begin{proof}
Assume the notation used in the definition of $S_k$.

We analyze two cases depending on the parity of $k$.
Suppose first that $k\ge 4$ is even.
For $0\le \ell \le k/2$, let $I_\ell = \{v_i^\ell\mid 1\le i\le s\}$.
For $0\le \ell \le k/2-2$, let $K_\ell = \{w_j^\ell\mid 1\le j\le r\}$.
Let $I' = I_{k/2}$ and $K'= V(G_k) \setminus I'$.
The vertices of $I'$ are mutually at distance $k+1$ in $S_k$, so they form an
independent set in $(S_k)^k$. The vertices of $K'$ are mutually at
distance at most $k$ in $S_k$, so they form a clique in $(S_k)^k$.
The vertices of $I'$ are at distance at most $k$ from the vertices
of $K'$ in $S_k$. It follows that the graph $(S_k)^k$ is a complete split graph. This settles the case $h = k$.
If $h>k$ then $(S_k)^h$ is a complete (and thus a complete split) graph.
For $h<k$, let us first consider the case when $h$ is even.
Then $I_{h/2}$ is an independent set in $(S_k)^h$, $K_{h/2-1}$
is a clique in $(S_k)^h$, and $v_i^{h/2}$ is adjacent to
$w_j^{h/2-1}$ in $(S_k)^h$ if and only if $v_i^{0}$ is adjacent to
$w_j^{0}$ in $S$. Thus, $(S_k)^h[I_{h/2}\cup K_{h/2-1}]$ is
isomorphic to $S$.
Consider now the case when $h$ is odd, $h > 1$. In this case, $I_{(h+1)/2}$
is an independent set in $(S_k)^h$, $K_{(h-3)/2}$ is a clique in
$(S_k)^h$, and $v_i^{(h+1)/2}$ is adjacent to $w_j^{(h-3)/2}$ in
$(S_k)^h$ if and only if $v_i^{0}$ is adjacent to $w_j^{0}$ in $S$.
Thus, $(S_k)^h[I_{(h+1)/2}\cup K_{(h-3)/2}]$ is isomorphic to $S$.

Suppose now that $k\ge 3$ is odd.
The sets $I_\ell$ and $K_\ell$ are defined analogously as in the case when $k$ is even (for appropriate ranges of $\ell$).
Similarly as above, it can be verified that for every $h\ge k$, the graph $(S_k)^h$ is a complete split graph.
For $h<k$, we have that the graph
$(S_k)^h[I_\ell\cup K_{\ell'}]$ is isomorphic to $S$ for a suitable choice of $\ell$ and $\ell'$, namely
for $(\ell,\ell') = (h/2,h/2-1)$ if $h\in \{1,\ldots, k-1\}$ is even, and
$(\ell,\ell') = ((h-1)/2,(h-1)/2)$ if $h\in \{1,\ldots, k-1\}$ is odd.
\end{proof}

For $k\ge 3$, let $\Sigma_k = \{S_k\mid S$ is a split graph$\}$.

\begin{theorem}\label{prop:unbounded-pi}
For every $k\geq 3$, $\pi(\Sigma _k) = k$.
\end{theorem}

\begin{proof}
This is an immediate consequence of Lemma~\ref{lem:unbounded-pi}, the fact that
split graphs are of unbounded clique-width~\cite{MR1726989},
and the fact that every complete split graph is a cograph and hence of clique-width at most $2$~\cite{MR1743732}.
\end{proof}

\begin{corollary}\label{cor:unbounded-pi}
For every $k\geq 1$, there exists a graph class ${\cal G}$ with $\pi({\cal G}) = k$.
\end{corollary}

\begin{proof}
For $k\ge 3$, the result follows from Theorem~\ref{prop:unbounded-pi}.
For $k = 1$, take ${\cal G}$ to be any graph class of bounded clique-width.
For $k = 2$, take ${\cal G}$ to be any graph class of graphs of diameter $2$ of unbounded clique-width
(for example, the class of graphs obtained from the class of grids by
adding to each grid a universal vertex).
\end{proof}}

\section{A sufficient condition for power-bounded clique-width}\label{sec:sufficient}

\begin{sloppypar}
By Proposition~\ref{prop:pw-bdd-cwd-powers}, every graph class of bounded
diameter is of power-bounded clique-width. We now generalize this observation
by giving a sufficient condition for power-boundedness of the clique-width
that is also applicable to graph classes of unbounded diameter.
\end{sloppypar}

\begin{theorem}\label{thm:sufficient-condition}
For every two positive integers  $k$ and $d$ and every class ${\cal G}$ of graphs of diameter at most $d$,
the class of graphs obtained from graphs in ${\cal G}$ by replacing
each of a set of at most $k$ edges with a path of length at least $1$
is of power-bounded clique-width.
\end{theorem}

To prove this theorem, we first state and prove a technical lemma.
Recall that two vertices $u$ and $v$ in a graph $G$ are said to be {\em twins} if
 $N(u)\setminus\{v\} =
N(v) \setminus \{u\}$. It is easy to see (and well known) that the twin relation is
an equivalence relation on $V(G)$, every equivalence class is either a
clique or an independent set, and for every two equivalence
classes, there are either all edges or no edges between them.
Thus, each equivalence class is a module, and
the quotient graph of $G$ with respect to the partition of $V(G)$ into twin classes
is well defined. A {\it $2$-path} in a graph $G$ is an induced path in
$G$ all the vertices of which are of degree $2$ in $G$.

\begin{lemma}\label{lem:sufficient}
For positive integers $k$ and $d$, let {${\cal G}(k,d)$} be the set
of all graphs $G$ that contain a set ${\cal P}$ of at most $k$
$2$-paths such that the diameter of $G-\cup_{P\in {\cal P}}V(P)$ is at most~$d$. Then, for every pair of
positive integers $k$ and $d$, the graph class {${\cal G}(k,d)$} is
of power-bounded clique-width.
\end{lemma}

\begin{proof}
Let $G\in {{\cal G}(k,d)}$ and let ${\cal P}$ be a set of $r$ $2$-paths in $G$ such that $r\le
k$ the diameter of the graph $G-X$ is at most $d$,
where $X:=\cup_{P\in {\cal P}}V(P)$.
We will show that $\cw(G^{d})\le p$ for some integer
$p$ depending only of $k$ and $d$, but not on $G$.
We may assume that $G$ is connected, that $X\neq \emptyset$ (otherwise $G$ is of bounded diameter, and
Proposition~\ref{prop:pw-bdd-cwd-powers}
applies), and that $G-X$ is non-empty (since paths are of clique-width at most $3$).
Let $C_1,\ldots, C_m$ denote the vertex sets of the components of $G-X$. Since $G$ is connected and $G[X]$ has at most $k$
components, each of which has neighbors in at most two components
of $G-X$, we have $m\le k+1$.

Let us analyze the structure of $G^d$. First, the assumption on the diameter implies that the subgraph of $G^d$ induced by each $C_i$ is
complete. For a vertex $x\in G-X$ and an endpoint $w$ of a component (path) in $G[X]$, let us define $$f(x,w) = \left\{
                                                                                        \begin{array}{ll}
                                                                                          \dist_G(x,w), & \hbox{if $\dist_G(x,w)\le d$;} \\
                                                                                         d+1, & \hbox{otherwise.}
                                                                                        \end{array}
                                                                                      \right.$$
Moreover, let $F(x)$ denote the array of values $f(x,w)$ for all endpoints $w$ of components of $G[X]$ (in some fixed order).
Clearly, this assignment of arrays to vertices in $G-X$ results in at most $(d+1)^{2k}$ different arrays.
Let us now define on each set $C_i$ an equivalence relation $\sim_i$ by the rule
$x\sim_i y$ if and only if $F(x) = F(y)$. Every such relation will have at most $(d+1)^{2k}$  equivalence classes.
Moreover, for every $i$, every two vertices $x,y\in C_i$ such that $x\sim_i y$
satisfy $N_{G^d}[x] = N_{G^d}[y]$; in particular, $x$ and $y$ are twins in $G^d$.
This can be proved by observing that for every vertex $z\in V(G)\setminus C_i$,
every shortest path from $x$ to $z$ goes through at least one endpoint $w$ of a path in ${\cal P}$, and using the fact that
if $\dist_G(x,z)\le d$ then $\dist_G(x,w)= \dist_G(y,w)$.

Proposition~\ref{prop:cwd-prime2} implies that $\cw(G^d) = \max\{\cw(H)\mid H$ is a prime induced subgraph of $G^d\}$.
In particular, since no prime induced subgraph of $G^d$ with at least three vertices contains a pair of twin vertices (as they would form a non-trivial module), this implies that the clique-width of $G^d$ equals the clique-width of the quotient graph $G'$ of
$G^d$ with respect to the partition of $V(G^d)$ into twin classes.
Hence, it is sufficient to show that the clique-width of $G'$ is bounded.

Let $U$ denote the set of vertices in $X$ that are in $G$ at distance at least $d+1$ from $V(G-X)$.
Let $Y'$ denote the set of twin classes of $G^d$
not containing any vertex of $U$.
Previous considerations imply that
{ $|Y'|\le (k+1)\cdot (d+1)^{2k}+2kd$}.
We may therefore assume that $U$ is non-empty (since otherwise $V(G') = Y'$ is of bounded size, and we are done).
Note that every vertex $u\in U$ has no neighbors in $G^d$ in $G-X$, and also no neighbors in any
component of $G[X]$ other than the component of $G[X]$ containing $u$. This implies that
the subgraph $G^d$ induced by $U$ is the $d$-th power of a disjoint union of paths.
Since $\cw(Q) = d+2$ if $Q$ is the $d$-th power of a path with at
least $(d+1)^2$ vertices~\cite{MR2972287}, this implies that the
clique-width of $G[U]$, and hence also of $G'-Y'$, is bounded from above by a function of $d$.

To complete the proof, recall that there exists a
function $g$ such that for every graph $H$ and every subset
$W\subseteq V(H)$, we have $\cw(H)\le g(\cw(H-W),|W|)$~\cite{MR2079015}. {Since $\vert Y'\vert$ is bounded},
this result implies that the clique-width of $G'$ is also
bounded by a function of $d$ and $k$.
\end{proof}

We are now ready to prove Theorem~\ref{thm:sufficient-condition}.

\begin{proof}[Proof of Theorem \ref{thm:sufficient-condition}]
Let $k$ and $d$ be two positive integers, and let $G$ be
a graph such that there exists a graph $H$ with the following properties:
\begin{itemize}
  \item[$(i)$] Every connected component of $H$ is of diameter at most $d$, and
  \item[$(ii)$] There exists a set $F\subseteq E(H)$ with $|F|\le k$ such that $G$ is the graph obtained from $H$ by replacing each edge $e\in F$ with a path of length at least~$2$.
\end{itemize}
Then, the graph $H' = H-F$ is of diameter at
most~$(k+1)d$~\cite{MR902717}. Thus, $G\in {{\cal G}(k,(k+1)d)}$
and the conclusion follows from Lemma~\ref{lem:sufficient}.
\end{proof}

Note that the result of Theorem~\ref{thm:sufficient-condition} is sharp, in the sense that neither of the two boundedness conditions can be dropped.
There exist graphs of unbounded diameter that are of power-unbounded clique-width, for example the class of grids. Moreover, if the number of subdivided edges is unbounded, then the resulting graph class can be of power-unbounded clique-width. Indeed, for every $k\ge 1$, let $G_{n,k}$ be the graph obtained from the complete graph $K_n$ by attaching to it ${n \choose 2}$ chordless paths of length $2k$, each connecting a different pair of vertices of $K_n$. Then, the $k$-th power of $G_{n,k}$ contains the graph $K_n^*$ as an induced subgraph, where $K_n^*$ denotes the graph obtained from a complete graph on $n$ vertices by gluing a triangle on every edge.
As shown in~{\cite{MR1726989}}, the clique-width of graphs $K_n^*$ is unbounded. Hence, for every $k\ge 1$, the family of graphs $\{G_{n,k}\mid n\ge 2\}$ is of power-unbounded clique-width.

\section{Hereditary graph classes of power-bounded and power-unbounded clique-width}\label{sec:hereditary}

{In this section, we develop several results related to power-boundedness of the clique-width in hereditary graph classes.
We start with a characterization for
graph classes
defined by a single forbidden induced subgraph.

}

\begin{theorem}\label{thm:monogenic}
For every graph $H$, the class of $H$-free graphs is of
power-bounded clique-width if and only if $H$ is a disjoint union
of paths.
\end{theorem}

\begin{proof}
{If $H$ has a cycle, then its girth is finite. Let $k$ be the girth of $H$.}
Then, every graph with girth at least $k+1$ is $H$-free. By
Example~\ref{cor:girth}, the class of graphs of girth at least
$k+1$ is of power-unbounded clique-width, hence the same holds
also for the larger class of $H$-free graphs.

{If $H$ is acyclic and $\Delta(H)\ge 3$, then} $H$ contains a claw as an
induced subgraph. Hence the class of claw-free graphs, and in
particular, the class of unit interval graphs, is a subclass of
$H$-free graphs. The power-unboundedness of the clique-width now
follows from Proposition~\ref{prop:path-powers}.

{If $H$ is acyclic and $\Delta(H)\le 2$ then $H$ is the disjoint union of paths,
thus an induced subgraph of a path. Hence, the $H$-free graphs are of
bounded diameter, and of power-bounded clique-width due to
Proposition~\ref{prop:pw-bdd-cwd-powers}.}
\end{proof}

{The case of forbidding two induced subgraphs instead of one turns out to be significantly more difficult.
In the rest of the section, we develop a complete characterization of the graph classes of the form ${\it Free}(\{H,H'\})$ that are of power-bounded clique-width, where $H$ and $H'$ are connected graphs. This is done in Section~\ref{subsec:hereditary}, using
results developed in Section~\ref{sec:sufficient} and in the next subsection.}

\subsection{A sufficient condition for power-unbounded clique-width {in hereditary classes}}\label{subsec:suff-pw-unbdd}

In this section, we adapt the approach from~\cite{MR2428563, MR2539362,MR3028197} to the notion of power-bounded
clique-width. For $i\ge 1$, let $H_i$ denote the graph depicted in Fig.~\ref{fig:H_i}.

\begin{figure}[!ht]
  \begin{center}
\includegraphics[height=22mm]{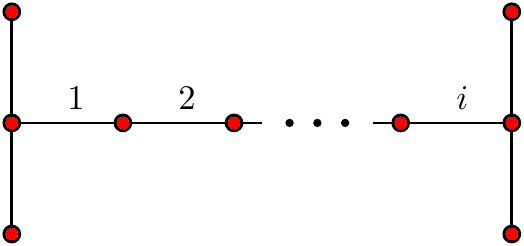}
  \end{center}
\caption{Graphs $H_i$.} \label{fig:H_i}
\end{figure}
For $k\ge 3$,
let us denote ${\cal S}_k:={\it Free}(\{K_{1,4},C_3,\ldots,C_k,H_1,\ldots,H_k\})$, and let
${\cal S}:=\bigcap\limits_{k\ge 3}{\cal S}_k$.
Note that a graph is in ${\cal S}$ if and only if every connected component of $G$ is of the form $S_{i,j,k}$
represented on the left in Fig.~\ref{fig:ST} (where the values of $i,j,k\ge 0$ may depend on component).

\begin{figure}[!ht]
  \begin{center}
\includegraphics[width=0.8\linewidth]{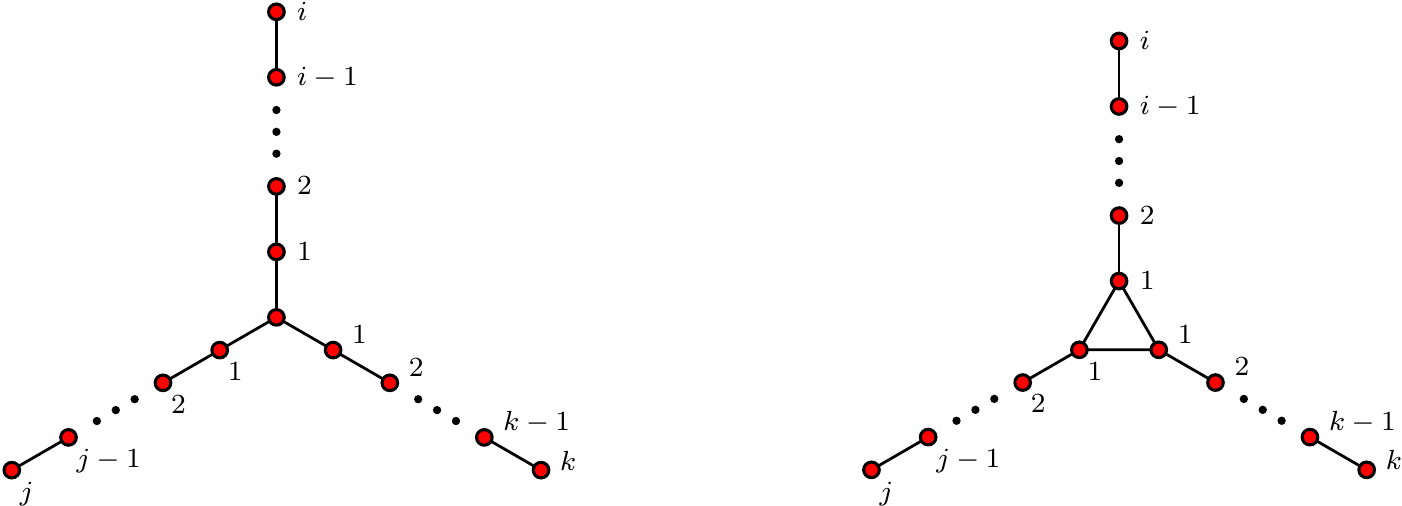}
  \end{center}
\caption{Graphs $S_{i,j,k}$ (left) and $T_{i,j,k}$ (right).}
\label{fig:ST}
\end{figure}

Denote the class of line graphs of graphs in ${\cal S}_k$ by ${\cal T}_k$.

\begin{proposition}\label{prop:Sk}
For every $k\ge 3$, ${\cal S}_k$ and  ${\cal T}_k$ are of power-unbounded clique-width.
\end{proposition}

\begin{proof}
The class ${\cal S}_k$ is of unbounded clique-width and of unbounded treewidth~\cite{MR2195310}.
Since graphs in ${\cal S}_k$ are $\{C_3, K_{1,4}\}$-free, every graph in ${\cal S}_k$ is of maximum degree at most $3$. Hence,
Proposition~\ref{prop:bdd-deg} applies, showing that ${\cal S}_k$ is of power-unbounded clique-width.

Gurski and Wanke~\cite{MR2362959} showed that for every graph $G$ and its
line graph $L(G)$, we have $(\tw(G)+1)/4\le \cw(L(G))\le
2\tw(G)+2$. This implies that ${\cal T}_k$ is of unbounded
clique-width. Since every graph in ${\cal S}_k$ is of maximum
degree at most $3$, every graph in ${\cal T}_k$ is of maximum
degree at most $5$. Hence, Proposition~\ref{prop:bdd-deg} implies
that the class ${\cal T}_k$ is of power-unbounded clique-width.
\end{proof}

To extend Proposition~\ref{prop:Sk} {to arbitrary hereditary graph classes}, let us recall the following two parameters,
introduced in~\cite{MR3028197}:
\begin{itemize}
  \item $\kappa(G)$ is the maximum $k$ such that $G\in {\cal S}_k$.
If $G$ belongs to no class ${\cal S}_k$, we define $\kappa(G)$ to
be $0$, and if $G$ belongs to all classes ${\cal S}_k$, then
$\kappa(G)$ is defined to be $\infty$. Also, for a set of graphs
${\cal G}$, we define $\kappa({\cal G}) = \sup\{\kappa(G)\mid G\in
{\cal G}\}$.
  \item $\lambda(G)$ is the maximum $\ell$ such that $G\in {\cal T}_\ell$.
If $G$ belongs to no {class} ${\cal T}_\ell$, then
$\lambda(G):=0$, and if $G$ belongs to every ${\cal T}_\ell$, then
$\lambda(G):=\infty$. For a set of graphs ${\cal G}$, we define
$\lambda({\cal G}) = \sup\{\lambda(G)\mid G\in {\cal G}\}$.
\end{itemize}

According to the definition, in order for $\kappa(G)$ to be infinite, $G$ must belong to every class ${\cal S}_k$,
that is, $G\in {\cal S}$. Moreover, \hbox{$\lambda(G)=\infty$} if and only if $G$ is the
line graph of a graph in $\cal S$. Let us denote the class of all
such graphs by ${\cal T}$. In other words, $\cal T$ is the class
of graphs every connected component of which has the form
$T_{i,j,k}$ represented on the right in Fig.~\ref{fig:ST} (where the values of $i,j,k\ge 0$ may depend on component).

The following result is implicit in the proofs of Theorems~2 and~{6} in~\cite{MR3028197}.

\begin{lemma}\label{lem:kappa-lambda}
Let ${\cal F}$ be a set of graphs. If $\kappa({\cal F})< \infty$,
then there is an integer $k$ such that ${\cal S}_k \subseteq {\it
Free}({\cal F})$. If $\lambda({\cal F})< \infty$, then there is an
integer $k$ such that ${\cal T}_k \subseteq {\it Free}({\cal F})$.
\end{lemma}

Proposition~\ref{prop:Sk} and Lemma~\ref{lem:kappa-lambda} imply the following.

\begin{theorem}\label{thm:kappa-lambda}
Let ${\cal F}$ be a set of graphs. If $\kappa({\cal F})< \infty$
or $\lambda({\cal F})<\infty$, then the class of ${\cal F}$-free
graphs is of power-unbounded clique-width.
\end{theorem}

\subsection{Graph classes defined by two connected forbidden induced subgraphs}\label{subsec:hereditary}

In this section, we prove our main result: a complete characterization of
graph classes of power-bounded clique-width within hereditary graph classes
defined by two connected forbidden induced subgraphs.

\begin{theorem}\label{thm:bigenic classes-characterization}
Let $A$ and $B$ be two connected graphs, and let ${\cal G}$ be the
class of $\{A,B\}$-free graphs. Then ${\cal G}$ is of
power-bounded clique-width if and only if either one of $A$ and
$B$ is a path, or one of $A$ and $B$ is isomorphic to some
$S_{1,j,k}$, and the other one to some $T_{1,j',k'}$ (represented in
Fig.~\ref{fig:ST}).
\end{theorem}

{We remark that Theorem~\ref{thm:bigenic classes-characterization}
implies the existence of classes of $\{A,B\}$-free graphs of power-bounded clique-width that are of
unbounded diameter and of unbounded clique-width.
An example of such a class is given by the class of} $\{$claw, bull$\}$-free graphs
(where the claw is the graph $S_{1,1,1}$, and the bull is the graph $T_{1,2,2}$).\footnote{The fact that
the class of claw-free bull-free graphs is of unbounded clique-width follows from the fact that it contains all
complements of triangle-free graphs (in particular, all complements of grids), hence
Proposition~\ref{prop:complementation} applies.}

Theorem~\ref{thm:bigenic classes-characterization} will be derived from the following two lemmas.
{It might be useful at this point to remind the reader of the simple
observation that ${\cal S}\cap {\cal T}$ equals the set of all disjoint unions of paths.}

\begin{lemma}\label{lem:bigenic classes-1}
Let $A$ and $B$ be two graphs, and let ${\cal G}$ be the class of
$\{A,B\}$-free graphs. Then, the following holds:
\begin{enumerate}[(i)]
  \item\label{item1} If $\{A,B\}\cap{\cal S} = \emptyset$ or $\{A,B\}\cap {\cal T} = \emptyset$, then
  ${\cal G}$ is of power-unbounded clique-width.
  \item\label{item3}  If $A\in {\cal S}\setminus {\cal T}$, $B\in {\cal T}\setminus {\cal S}$ and $A$ contains an induced $S_{2,2,2}$, then
  ${\cal G}$ is of power-unbounded clique-width.
  \item\label{item4}  If $A\in {\cal S}\setminus {\cal T}$, $B\in {\cal T}\setminus {\cal S}$ and $B$ contains an induced $T_{2,2,2}$, then
  ${\cal G}$ is of power-unbounded clique-width.
\end{enumerate}
\end{lemma}

\begin{proof}
\emph{\eqref{item1}} Suppose that $\{A,B\}\cap{\cal S} =
\emptyset$ or $\{A,B\}\cap {\cal T} = \emptyset$. Then
$\kappa(\{A,B\})<\infty$ or $\lambda(\{A,B\}) < \infty$, and by
Theorem~\ref{thm:kappa-lambda}, ${\cal G}$ is of power-unbounded
clique-width.

\emph{\eqref{item3}}
{If $A\in {\cal S}\setminus {\cal T}$, $B\in {\cal T}\setminus {\cal S}$ and $A$ contains an induced $S_{2,2,2}$, then
the class of $\{A,B\}$-free graphs contains the class of $\{S_{2,2,2}, T_{1,1,1}\}$-free graphs, which in turn
contains the class of bipartite permutation graphs (see, e.g.,~\cite{MR0221974,MR1861357,MR2071482}).} By
Proposition~\ref{prop:path-powers}, ${\cal G}$ is of
power-unbounded clique-width.

\emph{\eqref{item4}}
{If $A\in {\cal S}\setminus {\cal T}$, $B\in {\cal T}\setminus {\cal S}$ and $B$ contains an induced $T_{2,2,2}$, then
the class of $\{A,B\}$-free graphs contains the class of $\{S_{1,1,1}, T_{2,2,2}\}$-free graphs, which in turn
contains the class of unit interval graph graphs~\cite{MR0252267}.}
By Proposition~\ref{prop:path-powers}, ${\cal G}$ is of
power-unbounded clique-width.
\end{proof}

\begin{lemma}\label{lem:bigenic classes-3}
{For $k \ge 3$, let $G$ be a prime $\{S_{1,k,k}, T_{1,k,k}\}$-free graph.
Then, $G$ is obtained by subdividing a single edge in a graph of
bounded diameter.}
\end{lemma}

\begin{proof}
Let $G$ be a prime $\{S_{1,k,k}, T_{1,k,k}\}$-free graph.
In particular, $G$ is connected.

\medskip
\noindent {\it Claim 1. Let $u$ and $v$ be two vertices in $G$ with $\dist_G(u,v)\ge 2k$,
let $P$ be a shortest $u,v$-path in $G$, and let $Q$ be the subpath of $P$ induced by all vertices at distance at least
$k$ from each of the two endpoints of $P$.
Then, for every vertex $w\in N_G(Q)\setminus V(P)$, the neighborhood of $w$ in $P$ consists either of
three consecutive vertices, or of two vertices at distance two.}

\begin{proof}[Proof of claim]
Let  $w\in N_G(Q)\setminus V(P)$.
Due to the minimality of $P$, vertex $w$ cannot have two neighbors on $P$ at distance more than two.
Due to the $S_{1,k,k}$-freeness, vertex $w$ cannot have a single neighbor on $P$ (such a neighbor would belong to $Q$).
Due to the $T_{1,k,k}$-freeness, vertex $w$ cannot have only two consecutive neighbors on $P$.
Together, these observations prove the claim.
\end{proof}

\noindent {\it Claim 2. Let $u$, $v$ be a vertex pair with $\dist_G(u, v) = 2k +4$, and let $x$ be a vertex
with $\dist_G(u, x) = \dist_G(v, x) = k + 2$. Then, $d_G(x) = 2$.}

\begin{proof}[Proof of claim]
{Suppose for a contradiction that $d_G(x)\geq 3$.}
Let $P$ be a shortest $u,v$-path containing $x$, and let $x'$ be a neighbor of $x$ outside $P$.
By Claim 1, vertex $x'$ has two neighbors on $P$ at distance $2$, say $y$ and $z$.
By symmetry, we may assume that $\dist_G(u,y) <\dist_G(u,z)$.
Note that at least one of $y$ and $z$ is in $N_G[x]$, which implies that each of $y$ and $z$ is at distance at least $k$
from each of $u$, $v$. Let $A$ denote the set of common neighbors of $y$ and $z$ in $G$.
Then, $|A|\ge 2$. Since $G$ is prime, there exists a vertex, say $w$, in $V(G)\setminus A$,
that has both a neighbor, say $a$, and a non-neighbor, say $b$, in $A$.
Since $w\not\in A$, vertex $w$ is non-adjacent to either $y$ or $z$.
Applying Claim 1 to the shortest $u,v$-path, say $\tilde P$, induced by $(V(P)\setminus (N_G(y)\cap N_G(z)))\cup \{a\}$,
we infer that $w$ has a unique neighbor on $\tilde P$ at distance two from $a$.
Call this neighbor $y'$.
Suppose first that $y'$ is a neighbor of $y$.
Then, $w$ is not adjacent to $z$. But now, $(V(\tilde P)\cup\{b,w\})\setminus\{y\}$ induces a copy of
$S_{1,k_1,k_2}$ (centered at $z$) such that $k_1,k_2\ge k$,
{contradiction to} the $S_{1,k,k}$-freeness of $G$.
 The case when $y'$ is a neighbor of $z$, can be handled similarly.
\end{proof}

We split the rest of the proof into two cases.

\noindent{\bf Case 1:} {\it There exist two vertices, say $u$ and $v$, such that $d_G(u)\ge 3$, $d_G(v)\ge 3$
and $\dist_G(u,v)>7k+10$.}

Let $P$ be a shortest $u,v$-path.

\medskip
\noindent {\it Claim 3. For every $x\in V(G)\setminus V(P)$, we have $\dist_G(x,\{u,v\})\leq 2k+2$.}

\begin{proof}[Proof of claim] Suppose for a contradiction that there exists a vertex $x\in V(G)\setminus V(P)$ such that $\dist_G(x,\{u,v\})=2k+3$.

Without loss of generality, we may assume that $d(x,u)\le d(x,v)$.
Let $u'$ be the vertex of $P$ at distance $2k+3$ from $u$.
Let $P' = (x = v_0, v_1,\ldots, v_r = u')$ be a shortest $x,u'$-path in $G$ and let $u'' = v_i$ be the vertex in $V(P')\cap V(P)$ minimizing $i$.

We first show that $\dist_G(u,u'') \leq k+1$. Suppose for a
contradiction  $\dist_G(u,u'') \geq k+2$. Since $u''$ has degree at
least $3$ in $G$, Claim~2 ensures that $u''$ is at distance at
most $k+1$ from $v$. The length $\ell(P')$ of $P'$ can be bounded
from above as follows:
$$\ell(P')=\dist_G(x,u')\leq \dist_G(x,u)+\dist_G(u,u')=2k+3+2k+3=4k+6\,.$$
Consequently,
the length $\ell(P)$ of $P$ can be bounded from above as follows:
$$\ell(P)\le \dist_G(u,u')+\dist_G(u',u'')+\dist_G(u'',v)\le
2k+3+\ell(P')+k+1\le 7k+10\,,$$
a contradiction.

We claim that $u''$ is at distance more than $k+1$ from each endpoint of $P'$.
Indeed, $$\dist_G(u'',x)\geq \dist_G(u,x)-\dist_G(u,u'')\geq 2k+3-(k+1)=k+2$$ and
$$\dist_G(u'',u')\geq \dist_G(u,u')-\dist_G(u,u'')\geq 2k+3-(k+1)= k+2\,.$$
Now, since $P'$ is a shortest $x,u'$-path, Claim~2  implies that the degree of $u''$ in $G$ is equal to $2$, which contradicts the definition of $u''$.

{Note that} here we have used the assumption of Case 1 which guarantees that $u$---and hence also $u''$---is of degree at least $3$.
\end{proof}

Let $A$ denote the set of vertices at distance at most $2k+2$ from $\{u,v\}$.
Claim~3 implies that $G-A$ is a path. Moreover, since $G-A$ is a subpath of $P$,
every internal vertex of {$G-A$} is of degree $2$ in $G$, and
$G$ can be obtained from a graph of bounded diameter by subdividing one of its edges, {as desired}.

\medskip

\noindent{\bf Case 2:} {\it Every two vertices in $G$ of degree at least $3$ are at distance at most $7k+10$ from each other.}

If every vertex of $G$ has degree at most $2$, $G$ is a path or a cycle and hence $G$ is of clique-width at most $4$.
So we may assume that $G$ has a vertex, say $u$, of degree at least~$3$.

Let {$B$} be the set of vertices in $G$ at distance at most $8k+11$
from $u$. Then, $B$ will contain all vertices of $G$ of degree at
least $3$, together with all vertices that are at distance at most
$k+1$ from {some} vertex of degree at least $3$. In particular, the
subgraph $F$ of $G$ induced by $V(G)\setminus A$ consists only of
vertices of degree at most $2$ in $G$; in particular, $F$ is a
disjoint union of paths.

We claim that $F$ {is connected}. Suppose for a contradiction that $F$ {is disconnected}.
Let $s$ and $t$ be two vertices in different components of $F$. Then, any shortest path $P$ between $s$ and $t$ must pass through $B$, and since $B$ induces a connected graph, $P$ will contain a vertex, say $x$, of degree at least $3$.
However, this is a contradiction {to} Claim~2.

Thus, $G$ can be obtained from a graph of bounded diameter by subdividing one of its edges.
\end{proof}

\begin{proof}[Proof of Theorem \ref{thm:bigenic classes-characterization}]
{Let $A$ and $B$ be two connected graphs, and let ${\cal G}$ be the class of
$\{A,B\}$-free graphs.} Suppose that ${\cal G}$ is of power-bounded clique-width. By
Lemma~\ref{lem:bigenic classes-1}\eqref{item1}, we have
$A\in {\cal S}$ and $B\in {\cal T}$. We may assume that
neither of $A$ and $B$ is a path (otherwise, we are done).
Since $A$ and $B$ are connected, $A$ is of the form $S_{i,j,k}$ (for some
$i,j,k\ge 1$), and $B$ is of the form $T_{i',j',k'}$ (for some $i',j',k'\ge 1$).
By Lemma~\ref{lem:bigenic classes-1}\eqref{item3}, we have
that $A$ is of the form $S_{1,j,k}$ (for some $j,k$). Similarly,
Lemma~\ref{lem:bigenic classes-1}\eqref{item4} implies that
$B$ is of the form $T_{1,j',k'}$ (for some $j',k'$).

Suppose now that either one of $A$ and $B$ is a path, or one of
$A$ and $B$ is isomorphic to some $S_{1,j,k}$, and the other one
to some $T_{1,j,k}$. If one of $A$ and $B$ is a path, then
${\cal G}$ is of power-bounded clique-width by
Theorem~\ref{thm:monogenic}.
Otherwise, $A$ is $\{S_{2,2,2}, 2S_{1,1,1}\}$-free and $B$ is $\{T_{2,2,2}, 2T_{1,1,1}\}$-free. Since $A$ is $S_{2,2,2}$-free and $B$ is $T_{2,2,2}$-free, there exists a positive integer $k$ such that $A$ is an induced subgraph of $S_{1,k,k}$ and $B$ is an induced subgraph of $T_{1,k,k}$. { Thus, every graph in ${\cal G}$ is $\{S_{1,k,k},T_{1,k,k}\}$-free.
By Lemma~\ref{lem:bigenic classes-3}, every prime graph in ${\cal G}$ is obtained
from a graph of bounded diameter by subdividing a single edge.
Consequently, Theorem~\ref{thm:sufficient-condition} implies that the set of prime graphs in
${\cal G}$ is of power-bounded clique-width, and
hence ${\cal G}$ is of power-bounded clique-width,
by Proposition~\ref{lem:power-bdd-prime}.}
\end{proof}

\section{Discussion}\label{sec:conclusion}

We conclude the paper by mentioning several possibilities for future investigations related to the topics of this paper.

A main direction for future research is to perform a systematic study of graph classes
$(k,\ell)$-, $(k,\ast)$-, and $(\ast,\ell)$-power-bounded clique-width. Clearly, a graph class of $(\ast,\ell)$-power bounded clique-width is also of  $(\ast,\ell+1)$-power bounded clique-width. The converse fails in general, for instance for $\ell = 1$ (trivially) and also for $\ell = 2$ (for instance, the class of paths is $(1,3)$-power-bounded but not $(\ast,2)$-power-bounded). Theorem~\ref{prop:unbounded-pi} implies that
$(k+1,\ast)$-power-boundedness does not imply $(k,\ast)$-power-boundedness, for any value of $k\ge 1$.
We do not know whether $(k,\ast)$-power-boundedness implies $(k+1,\ast)$-power-boundedness.

Let us say that a graph class ${\cal G}$ of power-bounded clique-width is of {\em strongly power-bounded clique-width} if for every positive integer $k\ge \pi({\cal G})$, the class ${\cal G}^k$ is of bounded clique-width. Proposition~\ref{prop:pw-bdd-cwd-powers} implies that every graph class of bounded clique-width is of strongly power-bounded clique-width.

\begin{problem}\label{prob:strongly-pw-bdd-cwd}
Is it true that every graph class of power-bounded clique-width is also of strongly power-bounded clique-width?
\end{problem}

A positive answer to the above question would follow from a positive answer to the following one.

\begin{problem}
Is there a function $f$ such that for every graph $G$ and every positive integer $k$, we have $\cw(G^{k+1})\le f(\cw(G^k))$?
\end{problem}

On the other hand, a positive resolution to Problem~\ref{prob:strongly-pw-bdd-cwd} would imply a positive answer to the
following problem.

\begin{problem}\label{p3}
Is it true that every graph class ${\cal G}$ of power-bounded clique-width has only finitely many powers of unbounded
clique-width?
\end{problem}

Note that Proposition~\ref{prop:pw-bdd-cwd-powers} implies that for every positive integer $a$, a graph class of $(k,\ast)$-power-bounded clique-width is also of $(ak,\ast)$-power-bounded clique-width. Furthermore, for every graph class ${\cal G}$ for which we proved power-boundedness of the clique-width, our proofs in fact show that ${\cal G}$ has only finitely many powers of unbounded clique-width.

Since many interesting graph classes are hereditary, a closer understanding of the relation between the above notions and hereditary graph classes seems worth of study. For instance, given a hereditary graph class ${\cal G}$ of power-bounded clique-width, one could try to determine all pairs of integers $(k,\ell)$ such that ${\cal G}$ is of $(k,\ell)$-power-bounded clique-width. Furthermore,
what are the properties of the hereditary graph classes ${\cal C}_{k,\ell}$
for  $k\ge 1$ and $\ell\ge 1$, defined
by
$${\cal C}_{k,\ell} =
\{G\mid  \cw(H^k)\le \ell \textrm{ for each induced subgraph $H$
of $G$}\}\,?$$

\subsection*{Acknowledgements}

We are grateful to Nina Chiarelli for comments on an early version
and to the two anonymous referees for their very careful reading of the paper and their
pertinent and useful remarks that lead to an improved presentation of this work.

\begin{sloppypar}
This research was partially supported by the bilateral projects between Argentina and Slovenia, SLO/$11$/$12$ (resp.,~BI-AR/$12$--$14$--$012$), SLO/$11$/$13$ (resp.,~BI-AR/$12$--$14$--$013$). F.~Bonomo, L.~Grippo, and M.~D.~Safe were partially supported by
UBACyT Grant 20020100100980 and 20020130100808BA,
CONICET PIP 112-200901-00178 and 112-201201-00450CO, and
ANPCyT PICT-2012-1324 (Argentina).
M.~Milani\v c~was partially supported the Slovenian Research Agency (I$0$-$0035$, research program P$1$-$0285$ and research projects N$1$-$0032$, J$1$-$5433$, J$1$-$6720$, and J$1$-$6743$).
\end{sloppypar}

\bibliographystyle{abbrv}

\end{document}